\documentclass{birkjour}
 \newtheorem{thm}{Theorem}[section]
 \newtheorem{cor}[thm]{Corollary}
 \newtheorem{lem}[thm]{Lemma}
 \newtheorem{prop}[thm]{Proposition}
 \theoremstyle{definition}
 \newtheorem{defn}[thm]{Definition}
 \theoremstyle{remark}
 
 \newtheorem*{ex}{Example}
 \numberwithin{equation}{section}

\begin{document}

%-------------------------------------------------------------------------
% editorial commands: to be inserted by the editorial office
%
%\firstpage{1} \volume{228} \Copyrightyear{2004} \DOI{003-0001}
%
%
%\seriesextra{Just an add-on}
%\seriesextraline{This is the Concrete Title of this Book\br H.E. R and S.T.C. W, Eds.}
%
% for journals:
%
%\firstpage{1}
%\issuenumber{1}
%\Volumeandyear{1 (2004)}
%\Copyrightyear{2004}
%\DOI{003-xxxx-y}
%\Signet
%\commby{inhouse}
%\submitted{March 14, 2003}
%\received{March 16, 2000}
%\revised{June 1, 2000}
%\accepted{July 22, 2000}
%
%
%
%---------------------------------------------------------------------------
%Insert here the title, affiliations and abstract:
%

\title[Tanaka-Webster Biharmonic Hypersurfaces]
 {Tanaka-Webster Biharmonic Hypersurfaces in Sasakian Space forms}

%----------Author 1
\author[N. Mosadegh]{N. Mosadegh}

\address{Department of Mathematics,
 Azarbaijan Shahid Madani University,\\
Tabriz 53751 71379, Iran}

\email{n.mosadegh@azaruniv.ac.ir}

%\thanks{This work was completed with the support of our
%\TeX-pert.}
%----------Author 2
\author{E. Abedi}
\address{Department of Mathematics,
 Azarbaijan Shahid Madani University,\\
Tabriz 53751 71379, Iran}
\email{esabedi@azaruniv.ac.ir}
%----------classification, keywords, date


\date{January 1, 2004}
%----------additions
\dedicatory{}
%%% ----------------------------------------------------------------------

\begin{abstract}
In this article, we consider the concept of biharmonicity about hypersurfaces in the Sasakian space form which is equipped with the Tanaka-Webster connection. Then, we call them the Tanaka-Webster biharmonic hypersurfaces and obtain the necessary and sufficient existence condition about it. Also, we show a nonexistence result of the Thanaka-Webster biharmonic Hopf Hypersurfaces, where the gradient of the mean curvature is a principal direction. 

\end{abstract}
\footnotetext[1]{The second author is as corresponding author.}
%%% ----------------------------------------------------------------------
\maketitle
%%% ----------------------------------------------------------------------
%\tableofcontents
\section{Introduction}

A harmonic map $\psi: M\longrightarrow N$ between two Riemannian manifolds, where $M$ is compact, is known as the critical point of the energy functional
\begin{eqnarray*}
E:C^{\infty} (M,N)\longrightarrow R,\ \ E(\psi)=\frac{1}{2}\int_M \| d\psi\|^2 d\vartheta,
\end{eqnarray*}
where $C^{\infty} (M,N)$ denotes space of the smooth maps. With respect to the similar idea authors \cite{Ee, Ee1}, introduced $k-$harmonic maps and proposed they are the critical points of $E_k.$ Therefore, when $k=2$ the biharmonic maps represent as critical point of the bienergy
\begin{eqnarray*}
E_2:C^{\infty}(M, N)\longrightarrow R,\ \ E_2(\psi)=\frac{1}{2}\int_M |\tau (\psi)|^2d\vartheta,
\end{eqnarray*}
where the tension field associated to map $\psi$ is given by $\tau(\psi)= \textsf{trace} \nabla d\psi$. It is known, vanishing the tension field is a characterization of the harmonic maps. Later on, the first variation formula of $E_2$ was derived by Jiang \cite{Jiang} and given a new definition of $2$-harmonic maps in the variational point of view, written as 
\begin{eqnarray}\label{6.0}
\tau_2(\psi)= -J(\tau(\psi))= -\Delta \tau(\psi)- \textsf{trace} R^N(d\psi(.), \tau (\psi))d\psi(.)=0
\end{eqnarray}
that is, $\tau(\psi)\in ker J$, where $J$ is an elliptic differential operator, called the Jacobi operator. Here $\Delta=- \textsf{trace} \nabla^2$ stands for the Laplace-Beltrami operator, where $\nabla$ is induced connection in the pull back bundle $\psi^{-1}(TN)$. Also, $R^{N}$ is the curvature operator on $N$ which is defined by $R^N(X,Y)=[\nabla_X, \nabla_Y]- \nabla_{[X, Y]}$ for all $X$ and $Y$ tangent to $N$. Since, each harmonic map is biharmonic because $J$ is a linear operator, interesting is in non harmonic biharmonic maps which are named proper biharmonic.

Independently, by taking into account the harmonic mean curvature vector field, biharmonic notion of submanifold in the Euclidean space was defined by B. Y. Chen. Indeed, with respect to the characterization formula of the biharmonic Riemannian immersions into the Euclidean space, the biharmonic concept in the sense of Chen will be obtained, e.g, $\Delta H=0$ where $H$ denotes the mean curvature vector field (see \cite{Ch2}).

Into non-positive and positive curved spaces, nonexistence results for the proper biharmonic Riemannian immersions were obtained (see \cite{Ba1, Ba, ID, YU1, YU, Has}). Specialy, it was shown that does not exist a proper biharmonic hypersurfaces neither in the Euclidean space nor in the hyperbolic spaces $H^{n+1}$ base on the number of distinct principal curvatures of the Weingarten operator.

Additionally, in spaces of the nonconstant sectional curvature there exist several classification results concerning the proper biharmonic hypersurfaces which has been investigated in \cite{Oni,JI}. For example, all the proper biharmonic Hopf cylinders in $3$-dimensional Sasakian space forms were classified. Morevere, all the proper-biharmonic Hopf cylinders over a homogeneous real hypersurfaces in the complex projective spaces were determined. In particular, authors in \cite{Abedi2, Abedi3} got some results about the biharmonic immersed hypersurfaces in the warped product space as well.

 The aim of this paper is to introduce and study a new notion of the biharmonicity for the immersed hypersurfaces inside of a Sasakian space form, where is equipped with both the Tanaka-Webster and the Levi-Civita connections. At first, we find the necessary and sufficient condition of a Tanaka-Webster biharmonic hypersurface in the Sasakian space forms. Then, we consider the Tanaka-Webster biharmonic pseudo Hopf hypersurfaces and determine the existance and nonexistence results of them, where the gradient of the mean curvature, \textsf{grad}$|H|$, plays a significent role. Furthermore, we obtain the Tanaka-Webster biharmonic pseudo Hopf hypersurfaces which are minimal.
 
 \section{preliminaries}
In this section we introduce the notions and gather some known results that will be used throughout the paper.
Indeed, an odd dimensional manifold $M^{2m+1}$ equipped with tensor fields $\varphi$, $\xi$ and $\eta$ of type $(1, 1)$, $(0, 1)$ and $(1, 0)$, respectively, is called an almost contact manifold where the following condition satisfies
\begin{eqnarray*}
\varphi^2(X)=-X+ \eta(X)\xi,\ \ \
\eta(\xi)=1,\ \ \  \eta(\varphi X)=0,
\end{eqnarray*}
for $X \in T(M)$, also the triple $(\varphi, \xi, \eta)$ is named an almost contact structure. Now, $M^{2m+1}$ is endowed a Riemannian metric $g$, in such a way that
\begin{eqnarray*}
\eta(X)=g(\xi, X),\ \ \
 g(\varphi X, \varphi Y)= g(X, Y)- \eta(X)\eta(Y),\ \ \
\end{eqnarray*}
where $X$ and $Y$ are vector fields on $M^{2m+1}$. If $g(X, \varphi Y)=d\eta(X, Y)$, then $(\varphi, \xi, \eta, g)$ is called a contact metric structure. Now, $(\widetilde {M}^{2m+1},\varphi, \xi, \eta, g)$ is called a contact metric manifold. Also, a contact metric manifold $\widetilde{M}^{2m+1}$ is named a $K-$contact manifold, if $\xi$ be a Killing vector field. Then we have
\begin{eqnarray}\label{6.2}
\nabla_X \xi=-\varphi X,
\end{eqnarray}
where $\nabla$ is the Levi-Civita connection on $\widetilde{M}^{2m+1}$. A contact metric manifold $(\widetilde{M}^{2m+1}, \varphi, \xi, \eta, g)$  is known as a Sasakian manifold, if and only if
\begin{eqnarray}
 (\nabla_X \varphi)Y=g(X,Y)\xi- \eta(Y)X.
\end{eqnarray}
 A Sasakian manifold is a $k-$contact manifold as well \cite{Tan2}.

 Let $(\widetilde{M}^{2m+1}, \varphi, \xi, \eta, g)$ be a Sasakian manifold. The sectional curvature of $2$-plane spanned by $\{X, \varphi X\}$ is called $\varphi$-sectional curvature, where $X$ is orthogonal to $\xi$. Also, a Sasakian manifold which has constant $\varphi$-sectional curvature $c$ is called a Sasakian space form and determined by $\overline{M}^{2m+1}(c)$. The curvature tensor field of a Sasakian space form is given by
\begin{eqnarray}\label{66}
\overline{R}(X, Y)Z &=& -\frac{c-1}{4}\{\eta(Z)[\eta(Y)X - \eta(X)Y] \nonumber \\
& & + [g(Y, Z)\eta(X)- g(X, Z)\eta(Y)]\xi\nonumber \\
& & + g( \varphi X, Z)\varphi Y + 2g(\varphi X, Y)\varphi Z - g(\varphi Y,Z)\varphi X\}  \\
& & + \frac{c+3}{4}\{ g(Y, Z)X - g(X, Z)Y\}. \nonumber
\end{eqnarray}

The Tanaka-Webster connection can be defined on any strongly pseudoconvex CR-manifold. The Tanaka-Webster connection was introduced by Tanno at the end of $70's$ (see \cite{Ta1, We1})on any contact metric manifold, where the underlying almost CR-structure in general is not CR-integrable.  Precisely, the connection has been introduced for the contact metric manifolds, as following
\begin{eqnarray}\label{6.1}
\hat{\nabla}_X Y= \nabla_X Y+(\nabla_X \eta)(Y)\xi -\eta(Y)\nabla_X \xi + \eta(X)\varphi Y,
\end{eqnarray}
for all $X, Y\in \Gamma(T(\widetilde{M}^{2m+1}))$, where $\nabla$ denotes the Levi-Civita connection on $\widetilde{M}^{2m+1}$ see \cite{ D. E, Tan}. Furthermore, it was shown that a Tanaka-Webster connection $\hat{\nabla}$ is an unique linear connection on the contact metric manifold $\widetilde{M}^{2m+1}$, where the tensors $\xi, \eta$ and $g$ are all $\hat{\nabla}$-parallel, that is,
\begin{eqnarray}\label{6.66}
\hat{\nabla}\xi^{*}=0, \ \ \ \hat{\nabla} \eta=0,\ \ \ \hat{\nabla}g=0,
\end{eqnarray}
 and whose torsion tensors satisfies
\begin{eqnarray*}
\hat{T}(X, Y)= 2d\eta(X,Y)\xi, \ \ \ 
\hat{T}(\xi, \varphi X)= -\varphi \hat{T}(\xi, X),
\end{eqnarray*}
for all $X\in \Gamma(T(\widetilde{M}^{2m+1}))$.

At the end of this section, in order to illustrate the existence of the biharmonic hypersurfaces in the Sasakian space form, where $\xi$ is tangent to them, we construct an example.

%Now, we construct an example of a Riemannian biharmonic surface in Sasakian space form $R^3(-3)$ which is harmonic($H=0$),  :
 \begin{ex}
Let $R^{3}$ be a hypersurface in the Euclidean space $R^4$. Let $J$ be a standard almost complex structure in $R^4$ considered as ${C}^2$ and set $\xi=-J N$, where $N$ is an unit normal vector field of $R^3$. Define $\varphi$ by $\pi o J$, where $\pi$ is the natural projection of the tangent space of $R^4$ in to the tangent space of $R^3$. Let $(x, y, z)$ be the Euclidean coordinate in $R^3$, we consider
\begin{eqnarray*}
\eta= \frac{1}{2}(dz-ydx),\ \
g=\eta \otimes \eta +\frac{1}{4}(dx^2+ dy^2),\nonumber\\ \ \varphi(X\frac{\partial}{\partial x}+ Y\frac{\partial}{\partial y}+Z\frac{\partial}{\partial z})=Y\frac{\partial}{\partial x}-X\frac{\partial}{\partial y}+Yy\frac{\partial}{\partial z}
\end{eqnarray*}
where $\xi=2\frac{\partial}{\partial z}$. Then $(R^3, \varphi, \eta, \xi, g)$ is called a Sasakian space form where its $\varphi$-sectional curvature is $c=-3$. Let $f\in C^{\infty}(R^3(-3))$ defines $f(x,y,z)=x+z$, then we consider the level set of $f$ like $M^2=f^{-1}(0)= \{(x,y,z)\in R^3; x+z=0\}$ which is claimed as a minimal surface (as well as biharmonic) of $R^{3}(-3)$. In order to show this property, we choose an appropriate orthonormal frame field on $R^3(-3)$ such as
\begin{eqnarray*}
e_1= 2(\frac{\partial}{\partial x}+ y \frac{\partial}{\partial z}),\ \ e_2=-2\frac{\partial}{\partial y}\ \ ,e_3= 2\frac{\partial}{\partial z}
\end{eqnarray*}
then we calculate $\textsf{grad} f=\sum_{i=1}^3 e_i(f)e_i = 2((1+y)e_1+ e_3)$. So, $N=\frac{\textsf{grad} f}{|\textsf{grad }f|}=\frac{1}{\sqrt{(1+y)^2+1}}((1+y)e_1+e_3)$ is an unit normal vector on $M^2$. Also, $-\varphi N=V=-\frac{1+y}{2\sqrt{(1+y)^2+1}}e_2 $ is in $\Gamma(TM^2)$. Now, we take an orthonormal frame field $\{E_1=\frac{V}{|V|}=-e_2,\ \ E_2= \sqrt{\frac{(1+y)^2}{1+(1+y)^2}}(-\frac{1}{(1+y)}e_1+e_3)\}$ on $M^2$. Some easy computations show the following bracket relations, which we need to calculate the Weingarten operator $A$ of $M^2$ in the Sasakian space $R^3(-3)$, as following
\begin{eqnarray*}
&&[e_1,e_2]=2e_3,\ \ [e_1, e_3]=0,\ \ [e_2, e_3]=0\\
&&\nabla_{e_1}e_2=-\nabla_{e_2}e_1 =e_3,\ \ \nabla_{e_1}e_3=\nabla_{e_3}e_1=-e_2,\ \ \nabla_{e_2}e_3=\nabla_{e_3} e_2= e_1\nonumber
\end{eqnarray*}
after all, we can calculate
\begin{eqnarray*}
-AE_1= \nabla_{E_1} N=  \frac{1-(y+1)^2}{(1+y)^2+1}E_2 ,\\
-AE_2=\nabla_{E_2} N= \frac{1-(y+1)^2}{(1+y)^2+1}E_1,
\end{eqnarray*}
then we have
\begin{eqnarray*}
A=\left(
  \begin{array}{cc}
    0 & \frac{1-(y+1)^2}{(1+y)^2+1} \\
    \frac{1-(y+1)^2}{(1+y)^2+1} & 0 \\
  \end{array}
\right).
\end{eqnarray*}
So, the shape operator presents that the mean curvature $|H|=0$. In other words, the planes which are parallel to the $xz$-plane are biharmonic (harmonic) surfaces in the Sasakian space form $R^3(-3)$. Furthermore, the other non trivial example of the biharmonic surface in the Sasakian space form $R^{3}(-3)$ can be considered where $f\in C^{\infty}(R^3(-3))$ and $f(x,y,z)=x^2+z^2$, then we take the level set of $f$ like $M^2=f^{-1}(1)= \{(x,y,z)\in R^3; x^2+z^2=1\}\approx S^1 \times R$, similarly with respect to the above coordinate and orthonormal frame field, the cylinder $S^1\times R$ is a minimal surface (as well as biharmonic) in $R^{3}(-3)$ as well.
\end{ex}

\section{Tanaka-Webster Biharmonic Hypersurfaces }
Let $(\overline{M}^{2m+1}(c), \varphi, \xi, \eta, g)$ be a Sasakian space form with respect to constant $\varphi-$sectional curvature $c$, which is equipped with the Tanaka-Webster connection and  $M^{2m}$ is an isometrically immersed hypersurface there. We suppose that $\xi$ and $V=-\varphi N$ are tangent vector fields on $M^{2m}$, where $N$ is a local unit normal vector on $M^{2m}$.

Now, in order to have the Tanaka-Webster biharmonic notion of hypersurfaces in the Sasakian space form, we consider the following required Lemma.

\begin{lem}
Let $\overline{M}^{2m+1}(c)$ be a Sasakian space form. Then, the Tanaka-Webster connection holds in the following formula
\begin{eqnarray}\label{6.5}
\nabla^{\star}_X Y= \overline{\nabla}_X Y +g(X, \varphi Y)\xi+ \eta(Y)\varphi X +\eta(X)\varphi Y,
\end{eqnarray}
where $X, Y$ and $\overline{\nabla}$ denote tangent vector fields and the Levi-Civita connection on $\overline{M}^{2m+1}(c)$, respectively.
\end{lem}
\begin{proof}
By taking into the account the equations $(\ref{6.1})$ and $(\ref{6.2})$ we have
\begin{eqnarray}
\nabla^{\star}_X Y &=& \overline{\nabla}_X Y +[\overline{\nabla}_X \eta(Y)-\eta (\overline{\nabla}_X Y)]\xi+ \eta(Y)\varphi X+\eta(X)\varphi Y\nonumber\\
&=& \overline{\nabla}_X Y + g(Y, \overline{\nabla}_X \xi)\xi +\eta(Y)\varphi X+ \eta(X)\varphi(Y)\nonumber\\
&=& \overline{\nabla}_X Y + g(X, \varphi Y)\xi +\eta(Y)\varphi X+ \eta(X)\varphi(Y).\nonumber
\end{eqnarray}
\end{proof}
 Now, under an isometric immersion we can express
\begin{eqnarray}\label{6.42}
\tau_2^{\star}(\psi)= -\Delta^{\star} H- \textsf{trace} R^{\star}(d\psi(.), H)d\psi(.),
\end{eqnarray}
here $\Delta^{\star}$ stands for the Laplace-Beltrami operator on sections of the pull back bundle $\psi^{-1}(T(\overline{M}^{2m+1}))$ and $R^{\star}$ denotes the curvature tensor corresponding to $\nabla^{\star}$ on the Sasakian space form $\overline{M}^{2m+1}(c)$ and we utilize the following sign conventions
\begin{eqnarray*}
\Delta^{\star} X = - \textsf{trace} \nabla ^{\star^2} X, \ \ \ \ \forall X \in \psi^{-1}(T(\overline{M}^{2m+1})),\label{6.4}\\
 R^{\star}(X, Y)=[\nabla^{\star}_X, \nabla^{\star}_Y]- \nabla^{\star}_{[X, Y]}\ \ \ \ X, Y\in{T(\overline{M}^{2m+1})}.
\end{eqnarray*}
After all, we can define the Tanaka-Webster biharmonic hypersurfaces over a Sasakian space forms as following
\begin{defn}\label{6.3}
Let $\psi: M^{2m}\longrightarrow \overline{M}^{2m+1}(c)$ be an isometric immersion of a hypersurface $M^{2m}$ of the Sasakian space form $\overline{M}^{2m+1}(c)$ associated to the Tanaka-Webster connection $\nabla^{\star}$. Then, $M^{2m}$ is called Tanaka-webster biharmonic hypersurface if $\tau_2^{\star}(\psi)=0.$
\end{defn}
Now, we have all the necessary ingredients to prove the following results.
\begin{lem}\label{6.44}
Let $M^{2m}$ be an immersed hypersurface in a Sasakian space form $\overline{M}^{2m+1}(c)$ associated to the Tanaka-Webster connection $\nabla^{\star}$, isometrically. Then we have
\begin{eqnarray}
\Delta^{*} H = \Delta H + 2g(\textsf{grad}|H|,V)\xi-2\eta(\textsf{grad}|H|)V - 2H,
\end{eqnarray}
where $\xi$ and $V$ are tangent to $M^{2m}$.
\end{lem}
\begin{proof}
Let $\nabla^{\star}$ and $\overline{\nabla}$ denote the Tanaka-Webster and the Levi-Civita connections on $\overline{M}^{2m+1}(c)$, respectively. Also, let us denote by $\nabla$ the Levi-Civita connection on $M^{2m}$. We consider a parallel local orthonormal frame $\{e_\alpha\}_{\alpha=1}^{2m}$ at $p \in M^{2m}$. Then, from the equations $(\ref{6.5})$, $(\ref{6.4})$, with respect to this fact the tensors $\varphi, \eta$ and $g$ are all $\nabla^{\star}$-parallel then by applying the Weingarten equation $ \overline{\nabla}_{e_\alpha}N= -A e_\alpha + \nabla^{\perp}_{e_\alpha}N$ we have
 \begin{eqnarray}
\Delta^{\star} H &=& -\sum_{\alpha=1}^{2m} \nabla^{\star}_{e_\alpha}\nabla^{\star}_{e_\alpha}H\nonumber\\
&=& -\sum_{\alpha=1}^{2m} \nabla^{\star}_{e_\alpha}\big(\overline{\nabla}_{e_\alpha}H+ g(e_\alpha, \varphi H)\xi +\eta(H)\varphi e_\alpha+ \eta(e_\alpha)\varphi H\big)\nonumber\\
&=&-\sum_{\alpha=1}^{2m} \big\{\overline{\nabla}_{e_\alpha}\overline{\nabla}_{e_\alpha}H + g(e_\alpha, \varphi(\overline{\nabla}_{e_\alpha}H))\xi+ \eta(\overline{\nabla}_{e_\alpha}H)\varphi e_\alpha \nonumber\\ && + \eta(e_\alpha)\varphi(\overline{\nabla}_{e_\alpha}H)
 + g(e_\alpha, \varphi(\nabla^{\star}_{e_\alpha} H)\xi)+\eta(e_\alpha)\varphi(\nabla^{\star}_{e_\alpha}H)\big\}\nonumber\\
&=& \Delta H -2\big(\sum_{\alpha=1}^{2m} g(e_\alpha, e_\alpha |H|\varphi N\big)-\sum_{\alpha=1}^{2m} g(e_\alpha, \varphi A_He_\alpha))\xi\nonumber\\
 &&-2\big(\sum_{\alpha= 1}^{2m}\eta(e_\alpha)e_\alpha |H| \varphi N-
 \sum_{\alpha=1}^{2m}\eta(e_\alpha)\varphi A_H e_\alpha \big)\nonumber\\
   &&+\sum_{\alpha=1}^{2m}g(A_He_\alpha, \xi)\varphi e_\alpha +H.\label{6.6}
\end{eqnarray}
The next step is to compute all terms of the equation $(\ref{6.6})$ as following
\begin{eqnarray}\label{6.8}
\sum_{\alpha=1}^{2m} g(e_\alpha, e_\alpha |H|\varphi N)\xi&%&\sum_{\alpha=1}^{2m} g(e_\alpha, g(e_\alpha, \textsf{grad}|H|)\varphi N)\xi\nonumber\\
= - g(\textsf{grad}|H|, V)\xi,
\end{eqnarray}
and because the tensors $\varphi$ is skew symmetric then
\begin{eqnarray}\label{6.9}
&&\sum_{\alpha=1}^{2m}g(e_\alpha, \varphi A_H e_\alpha)\xi= \textsf{trace}(\varphi A_H)\xi
=0.
\end{eqnarray}
Also, for the other terms we have
\begin{eqnarray}\label{6.10}
\sum_{\alpha=1}^{2m}\eta(e_\alpha)e_\alpha |H|\varphi N &=& \sum_{\alpha=1}^{2m} \eta(e_\alpha)g(\textsf{grad}|H|, e_\alpha)\varphi N\nonumber\\
&=&\eta(\textsf{grad}|H|)\varphi N\nonumber\\
&=&-\eta(\textsf{grad}|H|)V.
\end{eqnarray}
Now, in order to calculate what is remainder we need to consider
\begin{eqnarray*}\label{6.39}
\sum_{\alpha=1}^{2m} \eta(\overline{\nabla}_{e_\alpha} H)\varphi e_\alpha&=& \sum_{\alpha=1}^{2m} g(\overline{\nabla}_{e_\alpha}H, \xi)\varphi e_\alpha
= \sum_{\alpha=1}^{2m}g(H, \varphi e_\alpha)\varphi e_\alpha,
\end{eqnarray*}
also 
\begin{eqnarray*}\label{6.40}
\sum_{\alpha=1}^{2m} \eta(\overline{\nabla}_{e_\alpha} H)\varphi e_\alpha &=& \sum_{\alpha=1}^{2m} g(\overline{\nabla}_{e_\alpha}H, \xi)\varphi e_\alpha
=-\sum_{\alpha=1}^{2m}g(A_He_\alpha , \xi)\varphi e_\alpha,
\end{eqnarray*}
so 
\begin{eqnarray}\label{6.50}
A_H \xi=-|H|V.
\end{eqnarray}
Then by taking into account the above equation the last two terms of  $(\ref{6.6})$ written as
\begin{eqnarray}\label{6.11}
\sum_{\alpha=1}^{2m} \eta(e_\alpha)\varphi A_H e_\alpha=\varphi A_H \xi= -\varphi |H|V
=-H,
\end{eqnarray}
and 
\begin{eqnarray}\label{6.12}
\sum_{\alpha=1}^{2m}g(A_H e_\alpha, \xi)\varphi e_\alpha &=& \sum_{\alpha, \beta=1}^{2m} g(A_H e_\alpha, \xi)(g(\varphi e_\alpha, e_\beta)e_\beta + g(\varphi e_\alpha, N)N)\nonumber\\
&=& -|H|\sum_{\alpha, \beta=1}^{2m} g(e_\alpha, V)(-g(e_\alpha,\varphi e_\beta)e_\beta - g(e_\alpha, \varphi N)N)\nonumber\\
&=&-|H|\sum_{\beta=1}^{2m} (g(V,\varphi e_\beta)e_\beta + g(V,V)N)\nonumber\\
&=&-H.
\end{eqnarray}
After all, from the equations $(\ref{6.8}),(\ref{6.9}),(\ref{6.10}),(\ref{6.11})$ and ($\ref{6.12}$) we have the result as it was claimed.
\end{proof}

\begin{lem}\label{6.43}
Let $\psi: M^{2m} \longrightarrow \overline{M}^{2m+1}(c)$ be an isometric immersion of $2m$-dimensional hypersurface $M^{2m}$ in a Sasakian space form $\overline{M}^{2m+1}(c)$. Let  the ambient manifold equipped with the $\nabla^{*}$ associated to the Tanaka-Webster connection, then
\begin{eqnarray}
\textsf{traceR}^{*}(d\psi(.),H)d\psi(.)=k H,
\end{eqnarray}
\end{lem}
where $k=\frac{1}{4}(15 - 6 m - c (3 + 2 m))$.
\begin{proof}
By applying the expression of the curvature tensor field where the ambient manifold is equipped with the Tanaka-Webster connection $\nabla^{\star}$, which satisfies in $(\ref{6.5})$, for $X \in T(\overline{M}^{2m+1}(c))$ we have
\begin{eqnarray}
R^{\star}(X, H)X &=& %\nabla^{\star}_{X}\nabla^{\star}_H X - \nabla^{\star}_{H}\nabla^{\star}_{X} X-\nabla^{\star}_{[X, H]}X.\nonumber\\
 \overline{R}(X, H)X-3g(X, \varphi H)\varphi X-\eta^2(X)\varphi^2 H.
\end{eqnarray}
 Now, we consider an appropriate local orthonormal frame field $\{e_\alpha\}_{\alpha=1}^{2m-1}\cup \{\xi, V\}$ on $\overline{M}^{2m+1}(c)$. It follows that
\begin{eqnarray}
\textsf{trace} R^{\star}(d\psi(.), H)d\psi(.)&=& \textsf{trace}\overline{R}(d\psi(.), H)d\psi(.)\nonumber\\
&&-3g(V, \varphi H)\varphi V-\eta^{2}(\xi)\varphi^2 H\nonumber\\
                                      &=&\sum_{\alpha=1}^{2m-1}\overline{R}(e_\alpha, H)e_\alpha\nonumber\\
                                      && +\overline{R}(\xi, H)\xi+ \overline{R}(V, H)V + 4H,\nonumber
\end{eqnarray}
then with respect to the equation $(\ref{66})$ a straightforward computation shows that
\begin{eqnarray*}
\overline{R}(e_\alpha, H)e_\alpha= -\frac{c+3}{4}g(e_\alpha, e_\alpha)H,
\end{eqnarray*}
and
\begin{eqnarray*}
\overline{ R}(\xi, H)\xi=-H,\ \ \overline{R}(V, H)V=-cH.
\end{eqnarray*}
Hence
\begin{eqnarray}\label{6.15}
\textsf{trace} R^{\star}(d\psi(.), H)d\psi(.)=-(2m-1)\frac{c+3}{4}H+(1-c)H.\nonumber
\end{eqnarray}
\end{proof}
After all we obtain the main result of this section
\begin{thm}\label{6.16}
Let $\psi: M^{2m} \longrightarrow \overline{M}^{2m+1}(c)$ be an isometric immersion of $2m-$ dimensional hypersurface $M^{2m}$ in the Sasakian space form $\overline{M}^{2m+1}(c)$ equipped with the $\nabla^{*}$ associated to the Tanaka-Webster connection. Then $M^{2m}$ is a Tanaka-Webster biharmonic hypersurface if and only if
\begin{eqnarray}
&& \left\{
     \begin{array}{ll}
       \hbox{$\Delta^{\perp}H = \textsf{trace} B((.), A_H(.)) +lH;$} \\\\
        \hbox{$\textsf{trace} A_{\nabla^{\perp}_{(.)} H}(.)+m |H| \textsf{grad}|H| + g(\textsf{grad}|H|, V)\xi- \eta(\textsf{grad} |H|)V=0$,}
     \end{array}
   \right.
\end{eqnarray}
where $l=\frac{-(2m+3)c-6m+7}{4}$ is constant, $B, A$ and $H$ denote the second fundamental form, the shape operator and the mean curvature vector field of $M^{2m}$ in $\overline{M}^{2m+1}(c)$, respectively. 
\end{thm}
\begin{proof}
By taking into account, the Definition $\ref{6.3}$, the Lemma $\ref{6.44}$ and Lemma $\ref{6.43}$  we have
\begin{eqnarray}\label{6.45}
\tau_2^{*}(\psi)= -(\Delta H + 2g(\textsf{grad}|H|,V)\xi-2\eta(\textsf{grad}|H|)V - 2H)-kH=0,
\end{eqnarray}
where
\begin{eqnarray}\label{6.13}
\Delta H &=& -\sum_{\alpha=1}^{2m}\overline{\nabla}_{e_\alpha}\overline{\nabla}_{e_\alpha} H\\
&=&-\Delta^{\perp} H + \textsf{trace}B(.,A_H .)+ \textsf{trace} A_{\nabla^{\perp}_{(.)}H}(.)+ \textsf{trace}\nabla A_H(.,.),\nonumber
\end{eqnarray}
in more details
\begin{eqnarray*}
\overline{\nabla}_{e_\alpha} \overline{\nabla}_{e_\alpha} H &=&\nabla^{\perp}_{e_\alpha}\nabla^{\perp}_{e_\alpha} H\nonumber \\ &&-A_{\nabla^\perp_{e_\alpha}H}e_\alpha-\nabla_{e_\alpha}A_H(e_\alpha)-B(e_\alpha, A_H(e_\alpha)),
\end{eqnarray*}
and
\begin{eqnarray*}\label{6.14}
\textsf{trace} \nabla A_H(.,.)&=& m  \textsf{grad}|H|^2\\
&&+ (\textsf{trace} \overline{R}(d\psi(.), H)d\psi(.))^\top +\textsf{trace} A_{\nabla^{\perp}_{(.)}H}(.)\nonumber.
\end{eqnarray*}
Also, the Lemma $\ref{6.43}$ yields that the tangent part of $ \textsf{trace} \overline{R}(d\psi(.), H)d\psi(.)$ Vanishes. Now, putting the above equations together, replacing them in the equation $(\ref{6.45})$, and splitting the normal and tangent part of it, then we reach the result as it was claimed.
\end{proof}
Regarding the mean curvature, from Theorem $\ref{6.16}.$ we can have the following result.

\begin{cor}\label{6.36}
Let $M^{2m}$ be a Tanaka-Webster biharmonic hypersurface with the constant mean curvature, then the sectional curvature of a Sasakian space form holds
\begin{eqnarray*}
c> \frac{-6m+7}{2m+3}.
\end{eqnarray*}
\end{cor}
\begin{proof}
Let $M^{2m}$ be a Tanaka-Webster biharmonic hypersurface in the Sasakian space form with the constant mean curvature $|H|= constant\neq 0$, then the Theorem $\ref{6.16}$ yields 
\begin{eqnarray*}
|B|^2=-l=-\frac{7-c(2m+3)-6m}{4},\nonumber
\end{eqnarray*}
which implies the result.
\end{proof}

\section{Tanaka-Webster biharmonic pseudo Hopf hypersurfaces }
 Let $x:( M^{2m}, g)\rightarrow (\overline{M}^{2m+1}(c),\overline{g})$ be an isometric immersion from a real hypersurface $M^{2m}$ in the Sasakian space form $\overline{M}^{2m+1}(c)$. We underline the ambient manifold is equipped with the Tanaka-Webster connection $\nabla^{\star}$ as well. Let $\overline{\nabla}$ and $\nabla$ denote the Levi-Civita connections on $\overline{M}^{2m+1}(c)$ and $M^{2m}$, respectively. We recall that $\xi$ and $V$ are tangent on $M^{2m}$. Then we have
%\begin{eqnarray}
$T(M^{2m})=D \oplus D^{\perp}$,
%\end{eqnarray}
where $D$ is a maximal $\varphi-$invariant distribution and $D^{\perp}= \textsf{Span}\{\xi, V \}$. Suppose that the  Weingarten operator $A$ satisfies $A D^\perp \subseteq D^\perp$ and $AD \subseteq D$. A hypersurface $M^{2m}$ is called a pseudo-Hopf hypersurface, provided that the Weingarten operator $A$ be invariant on $\textsf{Span}\{V, \xi\}$ (see \cite{Abedi}).
Supposed that $W_1 , W_2 \in \textsf{Span} \{\xi, V \}$ are eigenvectors of the Weingarten operator $A$ in which $AW_1=\gamma_1 W_1$ and $AW_2= \gamma_2 W_2$ where
\begin{eqnarray}\label{6.00}
& W_1=\xi \cos\theta + V \sin\theta,  \ \
& W_2=-\xi \sin\theta+ V \cos \theta
\end{eqnarray}
 for some $ 0 < \theta <\frac{\pi}{2}$, where $\gamma_1=-\tan \theta$ and $\gamma_2= \cot \theta$. Let, $AV= \alpha \xi+ \beta V$, then we have $\alpha=-1$ and $\beta= \frac{\cos 2\theta}{\cos \theta \sin \theta}$.

%Furthermore, here we taking into account the Codazzi equation as following
\begin{lem}\label{Am}
Let $(M^{2m}, g)$ be a hypersurface in a Sasakian space form $(\overline{M}^{2m+1}(c), \overline{g})$. Then the Coddazi equation for $X, Y$ and $Z$ tangent on $T(M^{2m})$ holds 
\begin{eqnarray}\label{6.20}
&&g((\nabla_X A)Y-(\nabla_Y A)X, Z)=\nonumber\\
&&-\frac{c-1}{4} \overline{g}\big(\overline{g}(\varphi X, Z)Y + 2\overline{g}(\varphi X, Y)Z- \overline{g}(\varphi Y, Z)X, V\big),
\end{eqnarray}
where $\nabla$ and $A$ are the Levi-Civita connection and the shape operator of $M^{2m}$.
\end{lem}
 %So, by taking into account the above assumption, we have the following
\begin{lem}\label{6.32}
Let $M^{2m}$ be a pseudo Hopf hypersurface of Sasakian space form $\overline{M}^{2m+1}(c)$. If the Weingarten operator $A$ for some $X\in D$ satisfies $AX=\lambda X$, then $\varphi X$ is an eigenvector corresponding to the eigenvalue which yields
\begin{eqnarray}
A\varphi X=\frac{2\beta \lambda +c+3}{4\lambda-2\beta}\varphi X.
\end{eqnarray}
\end{lem}
%where $\lambda$ and $c$ are eigenvalue of Weingarten operator corresponding to eigenvector $X$ and constant $\varphi$-sectional curvature, respectively.
\begin{proof}
Let $X, Y\in D$ be the eigenvectors of the Weingarten operator $A$. We consider $AV= -\xi+ \beta V$ and take the covariant derivative of both sides, then 
\begin{eqnarray*}
(\nabla_X A)V+ A\nabla_X V= X(\beta)V+ \beta \nabla_X V +\varphi X,
\end{eqnarray*}
where $\nabla$ denotes the Levi-Civita connection on $M^{2m}$ and $\nabla_X V=\varphi A X$ \cite{Abedi}. Hence, we get
\begin{eqnarray*}\label{6.18}
g((\nabla_X A)V, Y)+ g(A\varphi AX, Y)= \beta g(\varphi AX, Y)+g(\varphi X, Y),
\end{eqnarray*}
similarly
\begin{eqnarray*}\label{6.19}
g((\nabla_Y A)V, X)+ g(A\varphi AY, X)= \beta g(\varphi AY, X)+g(\varphi Y, X),
\end{eqnarray*}
so
\begin{eqnarray*}
g((\nabla_X A)Y- (\nabla_Y A)X, V)+2g(A\varphi AX, Y)= \nonumber\\ \beta g(\varphi AX, Y)
                                                       +\beta g(A\varphi X, Y)+2g(\varphi X, Y),\nonumber
\end{eqnarray*}
then by applying the Lemma \ref{Am} we obtain
\begin{eqnarray*}
-\frac{c-1}{2}g(\varphi X, Y)+2g(A\varphi AX, Y)=\nonumber\\ \beta g(\varphi AX, Y) 
                                                    +\beta g(A\varphi X, Y)+2g(\varphi X, Y),\nonumber
\end{eqnarray*}
consequently we have
\begin{eqnarray*}
(2\lambda-\beta)g(A\varphi X, Y)=(\beta \lambda +2 + \frac{c-1}{2})g(\varphi X,Y),
\end{eqnarray*}
so the claim is reached.
\end{proof}
Now we discuss about the Tanaka-Webster biharmonic pseudo Hopf hypersurfaces in more details. 
\begin{thm}
There exists no Tanaka-Webster biharmonic pseudo Hopf hypersurface in such away that $\textsf{grad} |H|$ is in the direction of the vector field in $D$. 
\end{thm}
\begin{proof}
We denote by $\psi: M^{2m}\rightarrow \overline{M}^{2m+1}(c)$ an isometric immersion where $M^{2m}$ is a Tanaka-Webster biharmonic pseudo Hopf hypersurface in a Sasakian space form $\overline{M}^{2m+1}(c)$. According to the assumption \textsf{grad}$|H|$ is in the direction of vectors in $D$. Then, by applying Theorem $\ref{6.16}$ directly, we get
\begin{eqnarray}
\left\{
  \begin{array}{ll}
   \hbox{$\Delta^{\perp}H = \textsf{trace} B((.), A_H(.)) +lH $;} \\
    \hbox{$\textsf{trace} A_{\nabla^{\perp}_{(.)} H}(.)+m |H|\textsf{grad} |H| =0,$}
  \end{array}
\right.
\end{eqnarray}
where $A \textsf{grad}|H|= -m|H| \textsf{grad} |H|$ is deduced by the second term. In other words \textsf{grad}$|H|$ is an eigenvector of the Weingarten operator $A$ corresponding to the eigenvalue $-m|H|$. Also, the previous Lemma $\ref{6.32}$ lets the Weingarten operator $A$ of a pseudo Hopf hypersurface $M^{2m}$ takes the following form with respect to a suitable orthogonal frame field $\{e_1, ...,e_{m-1}, e_m=\varphi e_1, ...,e_{2m-2}= \varphi e_{m-1}, e_{2m-1}=W_1, e_{2m}=W_2 \}$ in which
\begin{eqnarray}\label{6.21}
&&Ae_i= \lambda_i e_i, \ \ i=1, . . ., m-1\nonumber\\
&&A\varphi e_i= \overline{\lambda_i}\varphi e_i,  \ \ i= 1, ..., m-1\\
&&AW_1=- \gamma_1 W_1 ,\ \ \  AW_2=\gamma_2 W_2\nonumber
\end{eqnarray}
where, $\lambda_i $ and $\overline{\lambda_i}= \frac{2\beta \lambda_i +c+3}{4\lambda_i-2\beta}$ are the eigenvalues corresponding to the eigenvectors $e_i$ and $\varphi e_i$, respectively. We recall that $\gamma_1=-\tan \theta $ and $\gamma_2=\cot \theta $, consequently we get $\gamma_1 \gamma_2=-1$. Let $e_1=\frac{ \textsf{grad}|H|}{|\textsf{grad}|H||}$. Assume that \textsf{grad}$|H|$ is given by
% \begin{eqnarray*}
 $\textsf{grad}|H|= \sum_{i=1} ^{2m} e_i(|H|)e_i$.
%\end{eqnarray*}
 Then
\begin{eqnarray}\label{6.22}
 {e_1(|H|)\neq 0, \ \ \ \ \ \ \ \ \ e_i(|H|)=0, \ \ \ \ \ \ \ \ \ i=2, ..., 2m.}
\end{eqnarray}
 Also, it is  written
\begin{eqnarray}\label{6.23}
& \nabla_{e_i} e_j =\sum_{k=1} ^{2m} \omega_{ij} ^ke_k,
\end{eqnarray}
where $\omega_{ij} ^k$ is called Cartan coefficient. Then computing the compatibility condition $\nabla_{e_k} \langle e_i, e_j\rangle =0,$ which denotes
\begin{eqnarray}
\omega_{ki}^i &=& 0,\ \  i = j  \label{6.25}\\
\omega_{ki}^j+\omega_{kj}^i &=0& , \ \ i\neq j, \ \  i,j,k = 1,...,2m.\label{6.26}
\end{eqnarray}
 Morevere, from the Codazzi equation, and taking $(\ref{6.21})$ and $(\ref{6.23})$ we get
% \begin{eqnarray}
 %(\nabla_{e_k} A)e_i &=& (\nabla_{e_i}A)e_k\nonumber\\
 %\nabla_{e_k} Ae_i -A\nabla_{e_k}e_i &=& \nabla_{e_i}Ae_k -A\nabla_{e_i}e_k,\nonumber
 %\end{eqnarray}
 \begin{eqnarray}
 e_k(\lambda_i)e_i +(\lambda_i -\lambda_j)\omega_{ki}^je_j=e_i(\lambda_k)e_k +(\lambda_k-\lambda_j)\omega_{ik}^je_j,\nonumber
 \end{eqnarray}
 we multiply both sides of the above equation to $e_j$, then we arrive at the following equations
 \begin{eqnarray}
e_i(\lambda_j)&=&(\lambda_i- \lambda_j)\omega_{ji}^j\label{6.27} \\
(\lambda_i- \lambda_j)\omega_{ki}^j&=&(\lambda_k -\lambda_j)\omega_{ik}^j,\label{6.28}
 \end{eqnarray}
 for distinct $i,j,k=1,...,2m$. From $\lambda_1=-m|H|$ and $(\ref{6.22})$ we obtain
\begin{eqnarray}
e_1(\lambda_1) \neq 0, \ \ \ \ \ \ \ \  e_i(\lambda_1)=0,\ \ \ \ \ i=2,...,2m\label{6.29}
\end{eqnarray}
and
\begin{eqnarray}
%[e_i, e_j]\lambda_1=0, \ \ \ \ \ \ \ \ \ \ \ \  2 \leq i,j \leq 2m , i \neq j
 0=[e_i, e_j]\lambda_1 = (\nabla_{e_i} e_j - \nabla_{e_j} e_i)\lambda_1,\ \ \ 2 \leq i,j \leq 2m , i \neq j
\end{eqnarray}
which implies
\begin{eqnarray}
&\omega_{ij}^1 =\omega_{ji}^1, \label{6.30}	
\end{eqnarray}
 for distinct $i,j=2,..., 2m$.
It is claimed that, $\lambda_j \neq \lambda_1$ for $j=2,..., 2m$. Since, if $\lambda_j =\lambda_1$ for $j\neq 1$, utilize $(\ref{6.27})$ and put $i=1$
\begin{eqnarray*}
0= (\lambda_1- \lambda_j)\omega_{j1} ^j=e_1(\lambda_j)=e_1(\lambda_1),
\end{eqnarray*}
which contradicts to $(\ref{6.29})$. For $j=1$ and $ k, i\neq 1$ from $(\ref{6.28})$ we get
\begin{eqnarray}
 (\lambda_i -\lambda_1)\omega_{ki}^1=(\lambda_k - \lambda_1)\omega_{ik}^1
\end{eqnarray}
which together with $(\ref{6.30})$ yield
\begin{eqnarray}
& \omega_{ij}^1 =0, \ \ \ \ \ \ \ \ \ \ i\neq j, \ \ \ \ \ \ \ i,j= 2,... ,2m \label{6.31}
\end{eqnarray}
that combining with the equation $(\ref{6.26})$, implies $\omega_{i1}^j=0$, $i\neq j$ , $i,j= 2,...,2m$.\\

After all, by summarizing and considering the above equations and applying appropriate connections we will reach to a contradiction. Indeed, by using the equation $(\ref{6.23})$ we have %\begin{eqnarray}\label{2.11}
%& \nabla_{e_m} e_{2m} = \sum_{k=1}^{2m}\omega_{m 2m}^k e_k
%\end{eqnarray}
%where $\omega_{m 2m}^{2m}=0$ and $\omega_{m 2m}^{1}=0$ due to $(\ref{6.25})$ and $(\ref{6.31})$, respectively. On the other hand, with respect to $(\ref{6.00})$ and the assumption what $\varphi e_{i+1-m}=e_i$ for $i> m-1$, then we get
\begin{eqnarray}\label{6.41}
\nabla_{e_m} W_2 &=& \nabla_{e_m} (-\xi \sin\theta +V \cos \theta)\nonumber \\
&=& -e_m(\sin \theta)\xi - \sin (\theta) \nabla_{e_m} \xi + e_m(\cos \theta)V + \cos(\theta) \nabla_{e_m} V\nonumber \\
&=&-e_m(\sin \theta)(W_1 \sin \theta + W_2 \cos \theta)- \sin \theta (-\varphi e_m)\nonumber \\
 &+& e_m(\cos \theta)(W_1 \cos \theta - W_2 \sin \theta) + \cos \theta(\varphi Ae_m)\nonumber \\
&=&(- e_m (\sin \theta)\sin \theta +e_m(\cos \theta)\cos\theta)W_1 -(\sin \theta + \overline{ \lambda_1}\cos \theta)e_1\nonumber \\
&-&(e_m (\sin \theta ) \cos \theta +e_m (\cos \theta)\sin \theta)W_2.
\end{eqnarray}
On the one hand, from $(\ref{6.31})$ we can obtain the Cartan coefficient $\omega_{m 2m }^1=0 $, which is associated to the vector field $e_1$. Then, by taking this fact, from $(\ref{6.41})$ we get
\begin{eqnarray*}
& 0= \omega_{m 2m }^1= \sin \theta + \overline{\lambda_1}\cos \theta,
\end{eqnarray*}
which implies
\begin{eqnarray}\label{2.12}
\overline{\lambda_1} = -\tan \theta,
\end{eqnarray}
where $\overline{\lambda}_1$ is an eigenvalue of the Weingarten operator corresponding to the eigenvector $e_m=\varphi e_1$.

Similarly, by computing the $\nabla_{e_m} W_1$ and applying the equation $(\ref{6.31})$ we have
\begin{eqnarray*}
0= \omega_{m 2m-1}^1= \cos \theta -\overline{\lambda_1}\sin \theta, \ \
\end{eqnarray*}
 in which 
\begin{eqnarray}\label{2.13}
\overline{\lambda_1}= \cot \theta.
\end{eqnarray}
Finally, a contradiction is made by $(\ref{2.12})$ and $(\ref{2.13})$. Hence, the Tanaka-Webster biharmonic pseudo Hopf hypersurfaces do not exist in a Sasakian space forms whenever \textsf{grad}$|H|$ be an eigenvector of the Weingarten operator in the distribution $D$.
\end{proof}
Also, we can have
\begin{prop}
Let $M^{2m}$ be a Thanaka-Webster biharmonic pseudo Hopf hypersurface, where \textsf{grad}$|H|\in D^{\perp}$. Then $M^{2m}$ is either a minimal hypersurface $(|H|=0)$ or $|H|=-\frac{\gamma_1+\gamma_2}{m}$, 
where $\gamma_1=-\tan \theta$ and $\gamma_2= \cot \theta$ are the eigenvalues of the Weingarten operator.
\end{prop}
\begin{proof}
According to the assumption grad$|H|\in D^{\perp}$, so we can suppose that $\textsf{grad}|H|=\alpha \xi+ \beta V$. By using the Theorem $\ref{6.16}$ we have
\begin{eqnarray}\label{6.33}
\left\{
  \begin{array}{ll}
    \hbox{$\Delta^{\perp}H= \textsf{trace} B((.), A_H(.))+ lH$;} \\
     \hbox{$A \textsf{grad}|H|+m|H| \textsf{grad}|H|+\beta \xi-\alpha V=0$},
  \end{array}
\right.
\end{eqnarray}
%where $k=\frac{-(2m+3)c-6m+15}{4}$.
where the second terms yields
\begin{eqnarray}\label{6.35}
0=A(\alpha \xi+ \beta V)+ m|H|(\alpha \xi+ \beta V)+\beta \xi-\alpha V.
\end{eqnarray}
Now, by taking into account the equation $(\ref{6.00})$ we have
\begin{eqnarray*}
AW_1&=&\gamma_1W_1\nonumber\\
&=&\cos \theta A\xi+\sin \theta AV,
\end{eqnarray*}
so 
\begin{eqnarray}\label{6.34}
AV= -\xi+(\gamma_1+ \gamma_2)V,
\end{eqnarray}
then from the equations $(\ref{6.34})$ and $(\ref{6.35})$ we get
\begin{eqnarray*}
0=(m|H|\alpha)\xi+ (-2\alpha+ m|H|\beta+ \beta(\gamma_1 + \gamma_2))V,
\end{eqnarray*}
consequently it is obtained that
\begin{eqnarray*}
&&0= m|H|\alpha\\
&&0=-2\alpha+ m|H|\beta+ \beta(\gamma_1 + \gamma_2),
\end{eqnarray*}
those show, either $|H|=0$ or $\alpha=0$. When $|H|=0$ the Tanaka-Webster pseudo Hopf hypersurface $M^{2m}$ is minimal with respect to \textsf{grad}$|H|\in D^{\perp}$, clearly. Also, we get $|H|=-\frac{(\gamma_1 + \gamma_2)}{m}$, provided that $\alpha=0$. More precisely, where $\alpha=0$ then \textsf{grad}$|H|=\beta V$. By taking into account the Tanaka-Webster biharmonic condition of Theorem $\ref{6.16}$ we have
\begin{eqnarray*}
&&\Delta^{\perp}H= \textsf{trace} B((.), A_H(.))+ kH \nonumber\\
&& A \textsf{grad}|H|+m|H| \textsf{grad}|H|+\beta \xi=0,
\end{eqnarray*}
in which the second term, where \textsf{grad}$|H|=\beta V$, implies
\begin{eqnarray*}\label{6.38}
&&0=\beta AV+m|H|\beta V+\beta \xi,
\end{eqnarray*}
then by applying $(\ref{6.34})$ in the last equation we have
\begin{eqnarray*}
|H|=-\frac{\gamma_1 +\gamma_2}{m}.
\end{eqnarray*}
\end{proof}

\begin{prop}\label{6.37}
   The Tanaka-Webster biharmonic pseudo Hopf hypersurfaces have the constant mean curvature (CMC) provided that \textsf{grad}$|H|$ is in the direction $\xi$.
\end{prop}
\begin{proof}
Let $M^{2m}$ be a Tanaka-Webster biharmonic hypersurface in $\overline{M}^{2m+1}(c)$. We take \textsf{grad}$|H|=\alpha \xi$ according to the assumption. Then with respect to the Theorem $\ref{6.16}$ we have
\begin{eqnarray}
\left\{
  \begin{array}{ll}
   \hbox{$\Delta^{\perp}H= \textsf{trace} B((.),A_H(.))+ lH$;} \\
    \hbox{$A \textsf{grad}|H| + m|H| \textsf{grad}|H|-\alpha V=0$},
  \end{array}
\right.
\end{eqnarray}
where the second term of the above condition yields
\begin{eqnarray}
&&\alpha A\xi+ m|H|\alpha \xi-\alpha V=0,\nonumber
\end{eqnarray}
by taking into account the equation $(\ref{6.50})$ then $\alpha=0$. So, \textsf{grad}$|H|=0$ which states $|H|=$constant.
\end{proof}
Immediately the bellow result follows
\begin{cor}
Nonexistence Tanaka-Webster biharmonic pseudo Hopf hypersurface where \textsf{grad}$|H|$ is in the direction $\xi$ is deduced where $c<-3.$
\end{cor}
\begin{proof}
Obviously, it is followed by the Corollary $\ref{6.36}$ and Proposition $\ref{6.37}$.
\end{proof}

This work has been financially supported by the Azarbaijan Shahid Madani Uniyersity
under the grant number...

%\subsection{Author Data}

% ------------------------------------------------------------------------

\subsection*{Acknowledgment}
Many thanks to referee for developing this class file.

% ------------------------------------------------------------------------
\end{document}